\documentclass[12pt]{amsart}
\usepackage{amsfonts, amssymb}
\usepackage{parskip, fullpage, verbatim}
\usepackage[colorlinks=true]{hyperref}
\usepackage{breakurl}

\newcommand\CA{{\mathcal A}} 
\newcommand\CB{{\mathcal B}}

\newcommand\CF{{\mathcal F}}
 
\newcommand\CIF{{\mathcal {IF}}} 
\newcommand\CIFM{{\mathcal {IFM}}} 
\newcommand\CHIFM{{\mathcal {HIFM}}} 
\newcommand\CRFM{{\mathcal {RFM}}}

\newcommand\HIF{{\mathcal {HIF}}}
\newcommand\RF{{\mathcal {RF}}}

\newcommand\BBK{{\mathbb K}}

\newcommand\BBZ{{\mathbb Z}}

\newcommand\Ann{{\operatorname{Ann}}}
\newcommand\codim{\operatorname{codim}}

\newcommand\Der{{\operatorname{Der}}}

\newcommand\pdeg{\operatorname{pdeg}}

\numberwithin{equation}{section}

\theoremstyle{plain}
\newtheorem{lemma}[equation]{Lemma}
\newtheorem{theorem}[equation]{Theorem}

\newtheorem{corollary}[equation]{Corollary}
\newtheorem{proposition}[equation]{Proposition}
\theoremstyle{definition}
\newtheorem{defn}[equation]{Definition}
\newtheorem{remark}[equation]{Remark}
\newtheorem{example}[equation]{Example}

\thanks{We acknowledge 
support from the DFG-priority program 
SPP1489 ``Algorithmic and Experimental Methods in
Algebra, Geometry, and Number Theory''.}


\begin{document}

\title[ Inductive and Recursive Freeness of Localizations of Multiarrangements ]
{ Inductive and Recursive Freeness of Localizations of Multiarrangements} 

\author[T. Hoge]{Torsten Hoge}
\address
{Institut f\"ur Algebra, Zahlentheorie und Diskrete Mathematik,
Fakult\"at f\"ur Mathematik und Physik,
Leibniz Universit\"at Hannover,
Welfengarten 1,
30167 Hannover, Germany}
\email{hoge@math.uni-hannover.de}

\author[G. R\"ohrle]{Gerhard R\"ohrle}
\address
{Fakult\"at f\"ur Mathematik,
Ruhr-Universit\"at Bochum,
D-44780 Bochum, Germany}
\email{gerhard.roehrle@rub.de}

\author[A. Schauenburg]{Anne Schauenburg}
\email{anne.schauenburg@rub.de}


\allowdisplaybreaks

\begin{abstract}
The class of free multiarrangements
is known to be closed under taking localizations.
We extend this result to the stronger notions of inductive 
and recursive freeness.

As an application, we prove that 
recursively free multiarrangements 
are compatible with the product 
construction for multiarrangements.
In addition, we show how our results 
can be used to derive that
some canonical classes of free multiarrangements 
are not inductively free.
\end{abstract}

\maketitle


\section{Introduction}

The class of free arrangements plays a pivotal role 
in the study of hyperplane arrangements.
While an arbitrary subarrangement of a free arrangement
need not be free, freeness is retained by 
special types of subarrangements, so called localizations, 
\cite{terao:freefinitefields},
\cite[Thm.\ 4.37]{orlikterao:arrangements}.
It is natural to investigate this property for other 
classes of free arrangements.

For that purpose, let $\CF$, $\CIF$, $\RF$ and $\HIF$ denote the classes of 
free, inductively free, recursively free and 
hereditarily inductively free hyperplane arrangements, respectively
(see \cite[Defs.\ 4.53, 4.60]{orlikterao:arrangements}).
Note that we have proper inclusions throughout 
$\HIF \subsetneq \CIF \subsetneq \RF \subsetneq \CF$, see 
\cite[Ex.\ 2.16]{hogeroehrle:indfree}, 
\cite[Ex.\ 4.56]{orlikterao:arrangements}, and 
\cite[Rem.\ 3.7]{cuntzhoge}, respectively.
Our first result shows that localization
preserves each of these stronger notions of
freeness.

\begin{theorem}
\label{thm:local}
Each of the classes 
$\CIF$, $\RF$  and $\HIF$ 
is closed under taking localizations.
\end{theorem}

Moreover, 
freeness is compatible with 
the product construction for arrangements 
\cite[Prop.\ 4.28]{orlikterao:arrangements}.
It was shown in \cite[Prop.\ 2.10, Cor.\ 2.12]{hogeroehrle:indfree}
that this property also holds for both $\CIF$  and $\HIF$.
Our second main result extends this property to the class $\RF$.

\begin{theorem}
\label{thm:recfreeproducts}
A product of arrangements belongs to $\RF$ if and only if 
each factor belongs to $\RF$.
\end{theorem}

It can be a rather complicated affair to show that a given 
arrangement is or fails to be inductively free, see
for instance 
\cite[Lem.\ 4.2]{amendhogeroehrle:indfree},
\cite[\S 5.2]{cuntz:indfree}, and  
\cite[Lem.\ 3.5]{hogeroehrle:indfree}.  
In principle, one might have to search through
all possible chains of free subarrangements.
The notion of recursive freeness is even more 
elusive.
In that sense, Theorem \ref{thm:local} 
can serve to be a very useful tool in 
deciding that a given arrangement is not 
inductively or recursively free by 
exhibiting a small localization which 
is known to lack this property.

\medskip

In his seminal work \cite{ziegler:multiarrangements}, Ziegler 
introduced the notion of multiarrangements and initiated the study of their 
freeness.  
The question of freeness of 
multiarrangements 
is a very active field of research, 
e.g.~see \cite{yoshinaga:free14}.
In their ground breaking work 
\cite[Thm.\ 0.8]{abeteraowakefield:euler}, Abe, Terao and Wakefield
proved the 
Addition-Deletion Theorem for multiarrangements.

The class of free multiarrangements
is known to be closed under taking localizations,
see Theorem \ref{thm:multi}. 
Our third main result shows that 
localization also preserves the notions of 
inductive and recursive freeness in the setting of 
multiarrangements.
For this purpose, let 
$\CIFM$ and  $\CRFM$ 
denote the classes of
inductively free and recursively free 
multiarrangements, see
Definitions \ref{def:indfree} and \ref{def:recfree}.

\begin{theorem}
\label{thm:localmulti}
The classes $\CIFM$ and $\CRFM$ are
closed under taking localizations.
\end{theorem}

Theorem \ref{thm:local} follows for $\CIF$ and $\RF$ 
from 
Theorem \ref{thm:localmulti} as a special case, cf.~Remark \ref{rem:indfree}.

It follows from \cite[Lem.\ 1.3]{abeteraowakefield:euler} 
that a product of multiarrangements is free if and only if 
each factor is free. 
Armed with Theorem \ref{thm:localmulti}, we can readily 
extend this further to the classes
$\CIFM$ and $\CRFM$.

\begin{theorem}
\label{thm:recfreep}
A product of multiarrangements belongs to 
$\CIFM$ (resp.\ $\CRFM$) if and only if 
each factor belongs to $\CIFM$ (resp.\ $\CRFM$).
\end{theorem}

Theorem \ref{thm:recfreeproducts} follows from 
Theorem \ref{thm:recfreep} for $\CRFM$ as a special case.

\medskip

In Section \ref{sec:appl} we further demonstrate the versatility of 
Theorem \ref{thm:localmulti} by showing that 
certain multiarrangements stemming 
from complex reflection groups are not inductively 
free.
Among them are multiarrangements of a restricted arrangement
equipped with Ziegler's natural multiplicity 
on the restriction to a hyperplane, 
see Definition \ref{def:zieglermulti}.

For applications of 
Theorem \ref{thm:local} for $\RF$ and 
Theorem \ref{thm:recfreeproducts}
in the context of the 
classification of recursively free reflection arrangements,
see \cite{muecksch:recfree}.

\section{Recollections and Preliminaries}

\subsection{Hyperplane Arrangements}
\label{ssect:hyper}
Let $V = \BBK^\ell$ 
be an $\ell$-dimensional $\BBK$-vector space.
A \emph{hyperplane arrangement} is a pair
$(\CA, V)$, where $\CA$ is a finite collection of hyperplanes in $V$.
Usually, we simply write $\CA$ in place of $(\CA, V)$.
We write $|\CA|$ for the number of hyperplanes in $\CA$.
The empty arrangement in $V$ is denoted by $\Phi_\ell$.

The \emph{lattice} $L(\CA)$ of $\CA$ is the set of subspaces of $V$ of
the form $H_1\cap \dotsm \cap H_i$ where $\{ H_1, \ldots, H_i\}$ is a subset
of $\CA$. 
For $X \in L(\CA)$, we have two associated arrangements, 
firstly
$\CA_X :=\{H \in \CA \mid X \subseteq H\} \subseteq \CA$,
the \emph{localization of $\CA$ at $X$}, 
and secondly, 
the \emph{restriction of $\CA$ to $X$}, $(\CA^X,X)$, where 
$\CA^X := \{ X \cap H \mid H \in \CA \setminus \CA_X\}$.
Note that $V$ belongs to $L(\CA)$
as the intersection of the empty 
collection of hyperplanes and $\CA^V = \CA$. 
The lattice $L(\CA)$ is a partially ordered set by reverse inclusion:
$X \le Y$ provided $Y \subseteq X$ for $X,Y \in L(\CA)$.

If $0 \in H$ for each $H$ in $\CA$, then 
$\CA$ is called \emph{central}.
If $\CA$ is central, then the \emph{center} 
$T_\CA := \cap_{H \in \CA} H$ of $\CA$ is the unique
maximal element in $L(\CA)$  with respect
to the partial order.
We have a \emph{rank} function on $L(\CA)$: $r(X) := \codim_V(X)$.
The \emph{rank} $r := r(\CA)$ of $\CA$ 
is the rank of a maximal element in $L(\CA)$.
The $\ell$-arrangement $\CA$ is \emph{essential} 
provided $r(\CA) = \ell$.
If $\CA$ is central and essential, then $T_\CA =\{0\}$.
Throughout, we only consider central arrangements.

More generally, for $U$ an arbitrary subspace of $V$, we can define  
$\CA_U :=\{H \in \CA \mid U \subseteq H\} \subseteq \CA$, the 
\emph{localization of $\CA$ at $U$}, 
and 
$\CA^U := \{ U \cap H \mid H \in \CA \setminus \CA_U\}$,
a subarrangement in $U$.
The following observations are immediate
from the definitions, cf.~\cite[\S 2]{orlikterao:arrangements}.

\begin{lemma}
\label{lem:swap}
Let $\CB \subseteq \CA$ be a subarrangement and 
$Y \le X$ in  $L(\CA)$.
Then we have
\begin{itemize}
\item[(i)]
$\CB \cap \CA_X = \CB_X$; and
\item[(ii)]
$(\CB_X)^Y = (\CB^Y)_X$.
\end{itemize}
Note that $X$ and $Y$ need not be members of $L(\CB)$.
\end{lemma}

\subsection{Free Hyperplane Arrangements}
Let $S = S(V^*)$ be the symmetric algebra of the dual space $V^*$ of $V$.
If $x_1, \ldots , x_\ell$ is a basis of $V^*$, then we identify $S$ with 
the polynomial ring $\BBK[x_1, \ldots , x_\ell]$.
Letting $S_p$ denote the $\BBK$-subspace of $S$
consisting of the homogeneous polynomials of degree $p$ (along with $0$),
$S$ is naturally $\BBZ$-graded: $S = \oplus_{p \in \BBZ}S_p$, where
$S_p = 0$ in case $p < 0$.

Let $\Der(S)$ be the $S$-module of algebraic $\BBK$-derivations of $S$.
Using the $\BBZ$-grading on $S$, $\Der(S)$ becomes a graded $S$-module.
For $i = 1, \ldots, \ell$, 
let $D_i := \partial/\partial x_i$ be the usual derivation of $S$.
Then $D_1, \ldots, D_\ell$ is an $S$-basis of $\Der(S)$.
We say that $\theta \in \Der(S)$ is 
\emph{homogeneous of polynomial degree p}
provided 
$\theta = \sum_{i=1}^\ell f_i D_i$, 
where $f_i$ is either $0$ or homogeneous of degree $p$
for each $1 \le i \le \ell$.
In this case we write $\pdeg \theta = p$.

Let $\CA$ be an arrangement in $V$. 
Then for $H \in \CA$ we fix $\alpha_H \in V^*$ with
$H = \ker(\alpha_H)$.
The \emph{defining polynomial} $Q(\CA)$ of $\CA$ is given by 
$Q(\CA) := \prod_{H \in \CA} \alpha_H \in S$.

The \emph{module of $\CA$-derivations} of $\CA$ is 
defined by 
\[
D(\CA) := \{\theta \in \Der(S) \mid \theta(\alpha_H) \in \alpha_H S
\text{ for each } H \in \CA \} .
\]
We say that $\CA$ is \emph{free} if the module of $\CA$-derivations
$D(\CA)$ is a free $S$-module.

With the $\BBZ$-grading of $\Der(S)$, 
also $D(\CA)$ 
becomes a graded $S$-module,
\cite[Prop.\ 4.10]{orlikterao:arrangements}.
If $\CA$ is a free arrangement, then the $S$-module 
$D(\CA)$ admits a basis of $\ell$ homogeneous derivations, 
say $\theta_1, \ldots, \theta_\ell$, \cite[Prop.\ 4.18]{orlikterao:arrangements}.
While the $\theta_i$'s are not unique, their polynomial 
degrees $\pdeg \theta_i$ 
are unique (up to ordering). This multiset is the set of 
\emph{exponents} of the free arrangement $\CA$
and is denoted by $\exp \CA$.

Recall the class $\CIF$ of 
inductively free arrangements (\cite[Def.\ 4.53]{orlikterao:arrangements}).
There is an even stronger notion of freeness,
cf.\  \cite[\S 6.4, p.~253]{orlikterao:arrangements}.

\begin{defn}
\label{def:heredindfree}
The arrangement $\CA$ is called 
\emph{hereditarily inductively free} provided 
$\CA^X$ is inductively free for each $X \in L(\CA)$.
We abbreviate this class by $\HIF$.
\end{defn}

As $V \in L(\CA)$ and $\CA^V = \CA$, 
$\CA$ is inductively free, if it is 
hereditarily inductively free.
Also, $\HIF$ is a proper subclass of 
$\CIF$, see \cite[Ex.\ 2.16]{hogeroehrle:indfree}.

Let $U \subseteq V$ be a subspace of $V$.
Thanks to work of Terao, 
\cite[Prop.\ 5.5]{terao:generalizedexponents},  
\cite[Prop.\ 2]{terao:freefinitefields}, 
$\CA_U$ is free whenever $\CA$ is,
cf.~\cite[Thm.\ 1.7(i)]{ziegler:matroids},
\cite[Thm.\ 4.37]{orlikterao:arrangements},
or \cite[Prop.\ 1.15]{yoshinaga:free14}.

\subsection{Multiarrangements}
\label{ssec:multi}
A \emph{multiarrangement}  is a pair
$(\CA, \nu)$ consisting of a hyperplane arrangement $\CA$ and a 
\emph{multiplicity} function
$\nu : \CA \to \BBZ_{\ge 0}$ associating 
to each hyperplane $H$ in $\CA$ a non-negative integer $\nu(H)$.
Alternately, the multiarrangement $(\CA, \nu)$ can also be thought of as
the multiset of hyperplanes
\[
(\CA, \nu) = \{H^{\nu(H)} \mid H \in \CA\}.
\]

The \emph{order} of the multiarrangement $(\CA, \nu)$ 
is the cardinality 
of the multiset $(\CA, \nu)$; we write 
$|\nu| := |(\CA, \nu)| = \sum_{H \in \CA} \nu(H)$.
For a multiarrangement $(\CA, \nu)$, the underlying 
arrangement $\CA$ is sometimes called the associated 
\emph{simple} arrangement, and so $(\CA, \nu)$ itself is  
simple if and only if $\nu(H) = 1$ for each $H \in \CA$.

\begin{defn}
\label{def:submulti}
Let $\nu_i$ be a multiplicity of $\CA_i$ for $ i = 1,2$.
When viewed as multisets, suppose that 
$(\CA_1, \nu_1)$ is a subset of $(\CA_2, \nu_2)$.
Then we say that $(\CA_1, \nu_1)$ is a 
\emph{submultiarrangement} of $(\CA_2, \nu_2)$ and write 
$(\CA_1, \nu_1) \subseteq (\CA_2, \nu_2)$,
i.e.\ we have $\nu_1(H) \le \nu_2(H)$ for each $H \in \CA_1$.
\end{defn}

\begin{defn}
\label{def:localization}
Let $(\CA, \nu)$ be a multiarrangement in $V$ and let 
$U \subseteq V$ be a subspace of $V$. The 
\emph{localization of $(\CA, \nu)$ at $U$} is $(\CA_U, \nu_U)$,
where $\nu_U = \nu |_{\CA_U}$.
Note that for $X = \cap_{H \in \CA_U} H$, we have 
$\CA_X = \CA_U$ and $X$ belongs to the 
intersection lattice of $\CA$. 
\end{defn}

\subsection{Freeness of  multiarrangements}

Following Ziegler \cite{ziegler:multiarrangements},
we extend the notion of freeness to multiarrangements as follows.
The \emph{defining polynomial} $Q(\CA, \nu)$ 
of the multiarrangement $(\CA, \nu)$ is given by 
\[
Q(\CA, \nu) := \prod_{H \in \CA} \alpha_H^{\nu(H)},
\] 
a polynomial of degree $|\nu|$ in $S$.

The \emph{module of $\CA$-derivations} of $(\CA, \nu)$ is 
defined by 
\[
D(\CA, \nu) := \{\theta \in \Der(S) \mid \theta(\alpha_H) \in \alpha_H^{\nu(H)} S 
\text{ for each } H \in \CA\}.
\]
We say that $(\CA, \nu)$ is \emph{free} if 
$D(\CA, \nu)$ is a free $S$-module, 
\cite[Def.\ 6]{ziegler:multiarrangements}.

As in the case of simple arrangements,
$D(\CA, \nu)$ is a $\BBZ$-graded $S$-module and 
thus, if $(\CA, \nu)$ is free, there is a 
homogeneous basis $\theta_1, \ldots, \theta_\ell$ of $D(\CA, \nu)$.
The multiset of the unique polynomial degrees $\pdeg \theta_i$ 
forms the set of \emph{exponents} of the free multiarrangement $(\CA, \nu)$
and is denoted by $\exp (\CA, \nu)$.
It follows from Ziegler's analogue of Saito's criterion 
\cite[Thm.\ 8]{ziegler:multiarrangements} that 
\begin{equation*}
\label{eq:exp}
\sum \pdeg \theta_i = \deg Q(\CA, \nu) = |\nu|. 
\end{equation*}

Freeness for multiarrangements is 
preserved under localizations, 
\cite[Prop.\ 1.7]{abenuidanumata}.
The argument in the 
proof of \cite[Thm.\ 4.37]{orlikterao:arrangements}
readily extends to this more general setting.

\begin{theorem}
\label{thm:multi}
For $U \subseteq V$ a subspace, the 
localization $(\CA_U, \nu_U)$ of $(\CA, \nu)$ at $U$ is free 
provided $(\CA, \nu)$ is free.
\end{theorem}

Though constructive,
the proof of Theorem \ref{thm:multi} does not shed any 
light on the exponents of $(\CA_U, \nu_U)$ in relation to the
exponents of $(\CA, \nu)$. We do however have the following 
elementary observation. 

\begin{remark}
\label{rem:multiexponents}
Let $(\CA_1, \nu_1) \subseteq (\CA_2, \nu_2)$
be free multiarrangements with 
ordered sets of exponents
$\exp (\CA_i, \nu_i) = \{a_{i,1} \le \ldots \le a_{i,\ell}\}$
for $i = 1,2$.
Then $a_{1,j} \le a_{2,j}$ for each $1 \le j \le \ell$.
For, let $\{\theta_{i,1}, \ldots ,\theta_{i,\ell}\}$ 
be a homogeneous $S$-basis
of the free $S$-module $D(\CA_i, \nu_i)$
for $i = 1,2$.
For a contradiction, suppose that $k$ is the 
smallest index such that $a_{1,k} > a_{2,k}$. 
Then the grading of both $S$-modules and the fact that 
$D(\CA_2, \nu_2) \subseteq D(\CA_1, \nu_1)$ imply that 
$\theta_{2,1}, \ldots, \theta_{2,k} \in S\theta_{1,1} + \ldots + S \theta_{1,k-1}$.
But this shows that $\{\theta_{2,1}, \ldots ,\theta_{2,\ell}\}$ is
not algebraically independent over $S$, a contradiction.
\end{remark}

We recall a fundamental 
construction 
due to Ziegler, \cite[Ex.\ 2]{ziegler:multiarrangements}.

\begin{defn}
\label{def:zieglermulti}
Let $\CA$ be a simple arrangement.
Fix $H_0 \in \CA$ and  consider the restriction 
$\CA''$ with respect to $H_0$.
Define the \emph{canonical multiplicity} 
$\kappa$ on $\CA''$ as follows. For $Y \in \CA''$ set 
\[
\kappa(Y) := |\CA_Y| -1,
\]
i.e., $\kappa(Y)$ is the number of hyperplanes in $\CA \setminus\{H_0\}$
lying above $Y$.
Ziegler showed that freeness of $\CA$ implies 
freeness of $(\CA'', \kappa)$ as follows.
\end{defn}

\begin{theorem}
[{\cite[Thm.\ 11]{ziegler:multiarrangements}}]
\label{thm:zieglermulti}
Let $\CA$ be a free arrangement with exponents
$\exp \CA = \{1, e_2, \ldots, e_\ell\}$.
Let $H_0 \in \CA$ and consider the restriction 
$\CA''$ with respect to $H_0$.
Then the multiarrangement $(\CA'', \kappa)$ is free with
exponents
$\exp (\CA'', \kappa) = \{e_2, \ldots, e_\ell\}$. 
\end{theorem}

Note that 
the converse of Theorem \ref{thm:zieglermulti} is false.
For let $\CA$ be a non-free $3$-arrangement, 
cf.~\cite[Ex.\ 4.34]{orlikterao:arrangements}.
Since $\CA''$ is of rank $2$,
$(\CA'', \kappa)$ is free, 
\cite[Cor.\ 7]{ziegler:multiarrangements}.
Nevertheless, 
Ziegler's construction 
and in particular the question of a converse
of Theorem \ref{thm:zieglermulti}
under suitable additional hypotheses 
play an important role in the 
study of free simple arrangements, e.g.\
see
\cite[Thm.\ 2.1, Thm.\ 2.2]{yoshinaga:free04},
\cite{yoshinaga:free05},
\cite[Cor.\ 4.2]{abeyoshinaga},
\cite[Thm.\ 2]{schulze:free} and
\cite[Cor.\ 1.35]{yoshinaga:free14}.

\subsection{The Addition-Deletion Theorem for Multiarrangements}

We recall the construction from \cite{abeteraowakefield:euler}.

\begin{defn}
\label{def:Euler}
Let $(\CA, \nu) \ne \Phi_\ell$ be a multiarrangement. Fix $H_0$ in $\CA$.
We define the \emph{deletion}  $(\CA', \nu')$ and \emph{restriction} $(\CA'', \nu^*)$
of $(\CA, \nu)$ with respect to $H_0$ as follows.
If $\nu(H_0) = 1$, then set $\CA' = \CA \setminus \{H_0\}$
and define $\nu'(H) = \nu(H)$ for all $H \in \CA'$.
If $\nu(H_0) > 1$, then set $\CA' = \CA$
and define $\nu'(H_0) = \nu(H_0)-1$ and
$\nu'(H) = \nu(H)$ for all $H \ne H_0$.

Let $\CA'' = \{ H \cap H_0 \mid H \in \CA \setminus \{H_0\}\ \}$.
The \emph{Euler multiplicity} $\nu^*$ of $\CA''$ is defined as follows.
Let $Y \in \CA''$. Since the localization $\CA_Y$ is of rank $2$, the
multiarrangement $(\CA_Y, \nu_Y)$ is free, 
\cite[Cor.\ 7]{ziegler:multiarrangements}. 
According to 
\cite[Prop.\ 2.1]{abeteraowakefield:euler},
the module of derivations 
$D(\CA_Y, \nu_Y)$ admits a particular homogeneous basis
$\{\theta_Y, \psi_Y, D_3, \ldots, D_\ell\}$,
where $\theta_Y$ is identified by the 
property that $\theta_Y \notin \alpha_0 \Der(S)$
and $\psi_Y$ by the 
property that $\psi_Y \in \alpha_0 \Der(S)$,
where $H_0 = \ker \alpha_0$.
Then the Euler multiplicity $\nu^*$ is defined
on $Y$ as $\nu^*(Y) = \pdeg \theta_Y$.
Crucial for our purpose is the fact that the value $\nu^*(Y)$ 
only depends on the $S$-module $D(\CA_Y, \nu_Y)$.

Frequently, $(\CA, \nu), (\CA', \nu')$ and $(\CA'', \nu^*)$ 
is referred to as the \emph{triple} of multiarrangements with respect to $H_0$. 
\end{defn}

\begin{theorem}
[{\cite[Thm.\ 0.8]{abeteraowakefield:euler}}
Addition-Deletion-Theorem for Multiarrangements]
\label{thm:add-del}
Suppose that $(\CA, \nu) \ne \Phi_\ell$.
Fix $H_0$ in $\CA$ and 
let  $(\CA, \nu), (\CA', \nu')$ and  $(\CA'', \nu^*)$ be the triple with respect to $H_0$. 
Then any  two of the following statements imply the third:
\begin{itemize}
\item[(i)] $(\CA, \nu)$ is free with $\exp (\CA, \nu) = \{ b_1, \ldots , b_{\ell -1}, b_\ell\}$;
\item[(ii)] $(\CA', \nu')$ is free with $\exp (\CA', \nu') = \{ b_1, \ldots , b_{\ell -1}, b_\ell-1\}$;
\item[(iii)] $(\CA'', \nu^*)$ is free with $\exp (\CA'', \nu^*) = \{ b_1, \ldots , b_{\ell -1}\}$.
\end{itemize}
\end{theorem}

\begin{remark}
\label{rem:restriction}
We require a 
slightly stronger version of the restriction part of 
Theorem \ref{thm:add-del}, where we 
do not prescribe the exponents in \emph{loc.\ cit.} a priori.
Let  $(\CA, \nu), (\CA', \nu')$ and $(\CA'', \nu^*)$ 
be the triple with respect to a fixed hyperplane. 
It follows from \cite[Thm.\ 0.4]{abeteraowakefield:euler}
that if both $(\CA, \nu)$ and $(\CA', \nu')$ are free,
then their exponents are as given by parts (i) and (ii) in 
Theorem \ref{thm:add-del} 
(i.e., the exponents differing by 1 in one term is automatic,
cf.~\cite[Thm.\ 4.46]{orlikterao:arrangements}).
It then follows from the restriction part of \emph{loc.\ cit.}
that $(\CA'', \nu^*)$ is also free with exponents 
as in part (iii).
\end{remark}

Next we observe 
that localization is compatible 
with both deletion and restriction for multiarrangements. 

\begin{lemma}
\label{lem:euler}
Let $(\CA, \nu)$ be a multiarrangement, $X \in L(\CA)$, and $H \in \CA_X$.
Let $(\CA, \nu)$, $(\CA', \nu')$ and $(\CA'', \nu^*)$ be the triple 
with respect to $H$. Then we have 
\begin{itemize}
\item[(i)]
$((\CA_X)', (\nu_X)') = ((\CA')_X, (\nu')_X)$; and 
\item[(ii)]
$((\CA_X)'', (\nu_X)^*) = ((\CA_X)^H, (\nu_X)^*) = ((\CA^H)_X, (\nu^*)_X)
= ((\CA'')_X, (\nu^*)_X)$.
\end{itemize}
\end{lemma}

\begin{proof}
(i).
The proof follows easily 
from Definitions \ref{def:localization} and \ref{def:Euler}.

(ii).
Thanks to Lemma \ref{lem:swap}(ii), we have $(\CA_X)^H = (\CA^H)_X$.
By definition,  $(\nu^*)_X = \nu^*$ on $(\CA^H)_X$.
So it suffices to show that $(\nu_X)^*  = \nu^*$
on $(\CA_X)^H$. 
Let $Y \in (\CA_X)^H$. Then $Y = H \cap H'$ for some $H' \in \CA_X \setminus\{H\}$.
Consequently, $(\CA_X)_Y = \CA_Y$ and $(\nu_X)_Y = \nu_Y$.
Therefore, $D((\CA_X)_Y, (\nu_X)_Y) = D(\CA_Y, \nu_Y)$ and 
so by definition of the Euler multiplicity $(\nu_X)^* = \nu^*$,
as desired. 
\end{proof}

We recast Lemma \ref{lem:euler} in terms of triples as follows.

\begin{corollary}
\label{cor:euler}
Let $(\CA, \nu)$ be a multiarrangement, $X \in L(\CA)$, and $H \in \CA_X$.
Let $(\CA, \nu)$, $(\CA', \nu')$ and $(\CA'', \nu^*)$ be the triple 
of $(\CA, \nu)$ with respect to $H$. Then 
$(\CA_X, \nu_X)$, $((\CA')_X, (\nu')_X)$ and 
$((\CA'')_X, (\nu^*)_X)$ is the triple of 
 $(\CA_X, \nu_X)$ with respect to $H$.
\end{corollary}

In general, for $\CA$ a free hyperplane arrangement, $(\CA, \nu)$ need not be
free for an arbitrary multiplicity $\nu$, 
e.g.\ see \cite[Ex.\ 14]{ziegler:multiarrangements}.
However, for the following 
special class of multiarrangements this is always the case,
\cite[Prop.\ 5.2]{abeteraowakefield:euler}.

\begin{defn}
\label{def:single}
Let $\CA$ be a simple arrangement.
Fix $H_0 \in \CA$ and $m_0 \in \BBZ_{>1}$ and 
define the 
\emph{multiplicity $\delta$ concentrated at $H_0$}
by
\[
\delta(H) := \delta_{H_0,m_0}(H) := 
\begin{cases}
m_0 & \text{ if } H = H_0,\\
1   & \text{ else}.
\end{cases}
\]
\end{defn}

The following combines
\cite[Prop.\ 5.2]{abeteraowakefield:euler}, parts of its proof
and Theorem \ref{thm:zieglermulti}. Recall the definition of 
Ziegler's multiplicity $\kappa$ from Definition \ref{def:zieglermulti}.
The proof of Proposition \ref{prop:single}(i)
given in \cite{abeteraowakefield:euler} 
depends on Theorem \ref{thm:add-del}.
We present an elementary explicit 
construction for a homogeneous 
$S$-basis of $D(\CA, \delta)$.

\begin{proposition}
\label{prop:single}
Let $\CA$ be a free simple arrangement with
exponents $\exp \CA = \{1, e_2, \ldots, e_\ell\}$.
Fix $H_0 \in \CA$, $m_0 \in \BBZ_{>1}$ and let 
$\delta = \delta_{H_0,m_0}$ be 
as in Definition \ref{def:single}.
Let  $(\CA'', \delta^*)$ be the restriction of 
$(\CA, \delta)$ with respect to $H_0$. 
Then we have
\begin{itemize}
\item[(i)] 
$(\CA, \delta)$ is free with exponents
$\exp \CA = \{m_0, e_2, \ldots, e_\ell\}$;
\item[(ii)] 
$(\CA'', \delta^*) = (\CA'', \kappa)$ is free with exponents
$\exp  (\CA'', \kappa) = \{e_2, \ldots, e_\ell\}$.
\end{itemize}
\end{proposition}

\begin{proof}
(i).
We utilize the construction from 
the proof of \cite[Prop.\ 4.27]{orlikterao:arrangements}.
Let $\alpha_0 \in V^*$ with $H_0 = \ker \alpha_0$
and let $\Ann(H_0) = \{\theta \in D(\CA) \mid \theta(\alpha_0) = 0\}$
be the annihilator of $H_0$ in $D(\CA)$.
Let $\theta_E$ be the Euler derivation in $\Der(S)$ 
\cite[Def.\ 4.7]{orlikterao:arrangements}.
Then
\[
D(\CA) = S \theta_E \oplus \Ann(H_0)
\]
is a direct sum of $S$-modules.
Let $\{\theta_2, \ldots, \theta_\ell\}$ be a homogeneous
$S$-basis of $\Ann(H_0)$.
Then $\{\theta_E, \theta_2, \ldots, \theta_\ell\}$ is a homogeneous
$S$-basis of $D(\CA)$.
It follows that 
$\{\alpha_0^{m_0-1} \theta_E, \theta_2, \ldots, \theta_\ell\}$ is a homogeneous
$S$-basis of $D(\CA, \delta)$.

(ii). The equality $(\CA'', \delta^*) = (\CA'', \kappa)$
is derived as in the proof of 
\cite[Prop.\ 5.2]{abeteraowakefield:euler}.
The remaining statements then follow from 
Theorem \ref{thm:zieglermulti}. 
\end{proof}

\subsection{Inductive and Recursive Freeness for Multiarrangements}
\label{subsec:indutive}
As in the simple case, Theorem \ref{thm:add-del} motivates 
the notion of inductive freeness. 

\begin{defn}[{\cite[Def.\ 0.9]{abeteraowakefield:euler}}]
\label{def:indfree}
The class $\CIFM$ of \emph{inductively free} multiarrangements 
is the smallest class of arrangements subject to
\begin{itemize}
\item[(i)] $\Phi_\ell \in \CIFM$ for each $\ell \ge 0$;
\item[(ii)] for a multiarrangement $(\CA, \nu)$, if there exists a hyperplane $H_0 \in \CA$ such that both
$(\CA', \nu')$ and $(\CA'', \nu^*)$ belong to $\CIFM$, and $\exp (\CA'', \nu^*) \subseteq \exp (\CA', \nu')$, 
then $(\CA, \nu)$ also belongs to $\CIFM$.
\end{itemize}
\end{defn}

\begin{remark}[{\cite[Rem.\ 0.10]{abeteraowakefield:euler}}]
\label{rem:indfree}
The intersection of $\CIFM$ with the class of 
simple arrangements is $\CIF$.
\end{remark}

As for simple arrangements, if $r(\CA) \le 2$,
then $(\CA, \nu)$  is inductively free,  
\cite[Cor.\ 7]{ziegler:multiarrangements}.

\begin{remark}
\label{rem:indchain}
Suppose that $(\CA, \nu) \in \CIFM$. 
Then by Definition \ref{def:indfree} there exists a chain of 
inductively free submultiarrangements, starting with the 
empty arrangement 
\[
\Phi_\ell \subseteq (\CA_1, \nu_1) \subseteq (\CA_2, \nu_2) \subseteq \ldots \subseteq (\CA_n, \nu_n) = (\CA, \nu) 
\]
such that each consecutive pair obeys Definition \ref{def:indfree}(ii). 
In particular, $|\nu_i| = i$ for each $1 \le i \le n$ and so $|\nu| = n$.
Letting $H_i$ be the hyperplane in the $i$th inductive step, we have
 $(\CA_i'', \nu_i^*) = (\CA_i^{H_i}, \nu_i^*)$. 
In particular, $(\CA, \nu) = \{H_1, \ldots, H_n\}$ as a multiset.
So a fixed hyperplane may occur as one of the $H_i$ for different 
indices $i$.
We frequently refer to a sequence as above as 
an \emph{inductive chain} of $(\CA, \nu)$. 
\end{remark}

As in the simple case, 
Theorem \ref{thm:add-del} also motivates the notion of recursive 
freeness for multiarrangements, cf.~\cite[Def.\ 4.60]{orlikterao:arrangements}.

\begin{defn}
\label{def:recfree}
The class $\CRFM$ of \emph{recursively free} multiarrangements 
is the smallest class of arrangements subject to
\begin{itemize}
\item[(i)] $\Phi_\ell \in \CRFM$ for each $\ell \ge 0$;
\item[(ii)] for a multiarrangement $(\CA, \nu)$, if there exists a hyperplane $H_0 \in \CA$ such that both
$(\CA', \nu')$ and $(\CA'', \nu^*)$ belong to $\CRFM$, and $\exp (\CA'', \nu^*) \subseteq \exp (\CA', \nu')$, 
then $(\CA, \nu)$ also belongs to $\CRFM$;
\item[(iii)] 
for a multiarrangement $(\CA, \nu)$, 
if there exists a hyperplane $H_0 \in \CA$ such that both
$(\CA, \nu)$ and $(\CA'', \nu^*)$ belong to $\CRFM$, and 
$\exp (\CA'', \nu^*) \subseteq \exp (\CA, \nu)$, 
then $(\CA', \nu')$ also belongs to $\CRFM$.
\end{itemize}
\end{defn}

By Definitions \ref{def:indfree} and \ref{def:recfree}, 
$\CIFM \subseteq \CRFM$.

\begin{remark}
\label{rem:recchain}
Suppose that $(\CA, \nu) \in \CRFM$. 
It follows from Definition \ref{def:recfree} that there exists a chain of 
recursively free submultiarrangements, starting with the 
empty arrangement 
\[
\Phi_\ell \subseteq (\CA_1, \nu_1) \subseteq (\CA_2, \nu_2) \ldots (\CA_n, \nu_n) = (\CA, \nu) 
\]
such that each consecutive pair obeys Definition \ref{def:recfree}. 
In particular, $|\nu_i| = |\nu_{i-1}| \pm 1$ for each $1 \le i \le n$ and 
$|\nu| = n$.
We also refer to a sequence as above as 
a \emph{recursive chain} of $(\CA, \nu)$. 
\end{remark}

\subsection{Hereditary Inductive Freeness for Multiarrangements}
\label{subsec:heredindutive}

It is tempting to define the notion of a 
hereditarily inductively free multiarrangement simply 
by iterating the construction of the Euler multiplicity
from Definition \ref{def:Euler}.
However, the following two examples demonstrate that 
the resulting multiplicity on the restriction 
depends on the order in which the 
iteration is taking place.
The first is an instance of a constant 
multiplicity while the second is an example 
of a multiplicity concentrated at a single
hyperplane, cf.~Definition \ref{def:single}.
Thus, such a notion is only 
well-defined with respect to a
fixed total order on $\CA$.
We introduce such a notion 
in Definition \ref{def:extEuler}
below without further pursuing it seriously, 
because of its lack of uniqueness.

\begin{example}
\label{ex:heredindfree1}
Define the rank $3$ multiarrangement
$(\CA,\nu)$ by 
\[
Q((\CA,\nu)) = x^2y^2(x+y)^2(x+z)^2(y+z)^2.
\]
Let $H_1 := \ker x$, $H_2 := \ker (y+z)$ and $Y := H_1 \cap H_2$.
Then $(\CA^{H_1},\nu^{*_1}) = (\CA^{H_1},(3,2,2))$
 and $(\CA^{H_2},\nu^{*_2}) = (\CA^{H_2},(2,2,2,2))$, by
\cite[Prop.\ 4.1(6)]{abeteraowakefield:euler}. Moreover, we have
\[
  \left(\left( \CA^{H_1}\right)^{H_1\cap H_2},(\nu^{*_1})^{*_{12}}\right) =
   \left(\CA^Y,(3) \right) \not= \left( \CA^Y,(4)\right) =
   \left(\left( \CA^{H_2}\right)^{H_1\cap H_2},(\nu^{*_2})^{*_{12}}\right),
\]
according to \cite[Prop.\ 4.1(7), (6)]{abeteraowakefield:euler}.
\end{example}

\begin{example} 
\label{ex:heredindfree2}
Let $\CA = \CA(G(3,3,3))$ be the 
reflection arrangement of the 
unitary reflection group $G(3,3,3)$
with defining polynomial 
\[
Q(\CA(G(3,3,3))) = (x^3- y^3)(x^3- z^3)(y^3- z^3).
\]
Fix $H_1 := \ker (x-y) \in \CA$, $m_1 \in \BBZ_{>1}$ and let 
$\delta = \delta_{H_1,m_1}$ be 
as in Definition \ref{def:single}.
Let $H_2 := \ker (x-z)$ and 
$Y := H_1 \cap H_2$. Then we have
$(\CA^{H_1},\delta^{*_1}) = (\CA^{H_1},(2,2,2,2))$,
owing to 
\cite[Prop.\ 4.1(2)]{abeteraowakefield:euler}, and 
$((\CA^{H_1})^{H_1 \cap H_2},(\delta^{*_1})^{*_{12}}) = (\CA^Y,(4))$, 
by \cite[Prop.\ 4.1(6)]{abeteraowakefield:euler}. On the other hand
$(\CA^{H_2},\delta^{*_2}) = (\CA^{H_2},(m_1,1,1,1))$, thanks to 
\cite[Prop.\ 4.1(3)]{abeteraowakefield:euler} and 
$((\CA^{H_2})^{H_1\cap H_2},(\delta^{*_2})^{*_{12}}) = (\CA^Y,(3))$ 
for $m_1 \ge 3$, by \cite[Prop.\ 4.1(2)]{abeteraowakefield:euler}, 
and for $m_1 = 2$ by 
\cite[Prop.\ 4.1(4)]{abeteraowakefield:euler}. 
Therefore,
\[
\left(\left(\CA^{H_1}\right)^{H_1 \cap H_2},\left(\delta^{*_1}\right)^{*_{12}}\right) 
= \left(\CA^Y,(4)\right) \not= \left(\CA^Y,(3)\right) 
= \left(\left(\CA^{H_2}\right)^{H_1\cap H_2},\left(\delta^{*_2}\right)^{*_{12}}\right).
\]
\end{example}

In view of these examples, 
we extend the construction 
of a restriction of a multiarrangement to a hyperplane 
from Definition \ref{def:Euler} to 
restrictions of arbitrary members of $L(\CA)$ as follows.

\begin{defn}
\label{def:extEuler}
Fix a total order $\prec$ on $\CA$.
Let $Y\in L(\CA)$ be of rank $m$.
Then $\prec$ descends to give a total order on $\CA_Y$.
Then pick $H_1 \prec \ldots \prec H_m$ in $\CA_Y$ 
minimally with respect to $\prec$ such that 
$Y = H_1 \cap \ldots \cap H_m$.
One readily checks that 
\begin{equation}
\label{eq:iterate}
\CA^Y =
\left( \ldots \left((\CA^{H_1} )^{H_1\cap H_2}\right) \ldots \right)^{H_1\cap \ldots \cap H_m}. 
\end{equation}
Note that 
this is independent of the chosen order on $\{H_1, \ldots, H_m\}$.
Because of \eqref{eq:iterate},
if $\nu$ is a multiplicity on $\CA$, 
we can iterate the Euler multiplicity on 
consecutive restrictions in \eqref{eq:iterate} 
to obtain the \emph{restricted multiarrangement} on $\CA^Y$
with corresponding \emph{Euler multiplicity} 
which we denote again just by $\nu^*$ for simplicity
\[
(\CA^Y, \nu^*) := 
\left(\left( \ldots \left((\CA^{H_1} )^{H_1\cap H_2}\right) \ldots \right)^{H_1\cap \ldots \cap H_m} , \left( \ldots (\nu^*)^* \ldots \right)^*\right).
\]
\end{defn}

As demonstrated in 
Examples \ref{ex:heredindfree1} and \ref{ex:heredindfree2}, 
the construction of $(\CA^Y, \nu^*)$ in Definition \ref{def:extEuler}
depends on the chosen order of the iterated Euler multiplicities.
Nevertheless, using Lemma \ref{lem:euler}(ii) repeatedly, 
we get compatibility of 
restricted multiarrangements 
with taking localizations.

\begin{corollary}
\label{cor:extEuler}
Fix an order on $\CA$.
Let $(\CA, \nu)$ be a multiarrangement, $X \in L(\CA)$ and $Y \in L(\CA_X)$.
Then with the notation as in Definition \ref{def:extEuler}, 
we have 
\[
((\CA_X)^Y, (\nu_X)^*) = ((\CA^Y)_X, (\nu^*)_X).
\]
\end{corollary}

\begin{defn}
\label{def:heredindfree}
Fix an order on $\CA$.
The  multiarrangement $(\CA, \nu)$ is called 
\emph{hereditarily inductively free 
(with respect to the order on $\CA$)} provided 
$(\CA^Y, \nu^*) $ is inductively free for every  
$Y\in L(\CA)$. We abbreviate this class by 
$\CHIFM$.
\end{defn}

Clearly, $\CHIFM \subseteq \CIFM$.
With the aid of Corollary \ref{cor:extEuler}
and Theorem \ref{thm:localmulti} for $\CIFM$, 
one can extend the latter to the class $\CHIFM$.
Also, using Theorem \ref{thm:recfreep},  
one readily obtains the compatibility of 
$\CHIFM$ with 
the product construction for multiarrangements.
We leave the details to the interested reader.

\section{Proofs of Theorems \ref{thm:localmulti}
and \ref{thm:recfreep}}

\subsection{Inductive and Recursive Freeness of Localizations of Multiarrangements}
\label{ssec:localizations}

The following is a reformulation of Theorem \ref{thm:localmulti}.

\begin{theorem}
\label{thm:multitwo}
Let $U \subseteq V$ be a subspace and let 
$(\CA, \nu)$ be a multiarrangement in $V$.
\begin{itemize}
\item[(i)] If $(\CA, \nu)$ is  inductively free,
then so is the localization $(\CA_U, \nu_U)$.
\item[(ii)] If $(\CA, \nu)$ is recursively free,
then so is the localization $(\CA_U, \nu_U)$.
\end{itemize}
\end{theorem}

\begin{proof}
We readily reduce to the case where 
we localize with respect to a space
$X$ belonging to the 
intersection lattice of $\CA$. 
For, letting $X = \cap_{H \in \CA_U} H \in L(\CA)$,
 we have $\CA_X = \CA_U$.

(i).
We argue by induction on the rank $r(\CA)$.
If $r(\CA) \le 3$, then 
$r(\CA_X) \le 2$ for $X \ne T_\CA$, so
the result follows thanks to 
\cite[Cor.\ 7]{ziegler:multiarrangements}.

So suppose $(\CA, \nu)$ is inductively free of rank $r > 3$
and that the statement holds for 
all inductively free multiarrangements of rank less than $r$.

Since $(\CA, \nu)$ is inductively free, there is an
inductive chain  $(\CA_i, \nu_i)$  of $(\CA, \nu)$,  
where $|\nu_i| =i$, for $i = 1, \ldots, n = |\nu|$, 
see Remark \ref{rem:indchain}.
Then thanks to Lemma \ref{lem:swap}(i), we have
\begin{equation}
\label{eq:seq1}
\CA_X \cap \CA_i = (\CA_i)_X.
\end{equation}
For $H \in \CA_X \cap \CA_i$, we have $H \le X$, and so by \eqref{eq:seq1} and Lemma \ref{lem:swap}(ii),  
\begin{equation}
\label{eq:seq2}
\left(\CA_X \cap \CA_i\right)^{H} = \left((\CA_i)_X\right)^{H} 
= \left(\CA_i^{H}\right)_X.
\end{equation}

Consequently, localizing each member of 
the sequence $(\CA_i, \nu_i)$ at $X$, removing redundant 
terms if necessary and reindexing the resulting distinct multiarrangements, 
we obtain the following sequence of submultiarrangements of 
$(\CA_X, \nu_X)$, 
\begin{equation}
\label{eq:seq3}
(\CA_{1,X}, \nu_{1,X}) \subsetneq (\CA_{2,X}, \nu_{2,X})
\subsetneq  \ldots \subsetneq (\CA_{m,X}, \nu_{m,X}) = (\CA_X, \nu_X),
\end{equation}
where $\CA_{i,X}$ is short for $(\CA_i)_X$ and 
$\nu_{i,X}$ for $\nu_i|_{(\CA_i)_X}$.
In particular, $|\nu_{i,X}| = i$ and 
$m = |\nu_X|$.
We claim that \eqref{eq:seq3} is an inductive chain of $(\CA_X, \nu_X)$.

Now let $H_i \in \CA_X \cap \CA_i = \CA_{i,X}$ be the relevant 
hyperplane in the $i$th step in the sequence \eqref{eq:seq3}.
Let $(\CA_{i,X}, \nu_{i,X})$, $(\CA_{i,X}', \nu_{i,X}')$
and $(\CA_{i,X}'', \nu_{i,X}^*)$
be the triple with respect to $H_i$.

Note that, since 
$(\CA_{i-1, X}, \nu_{i-1, X}) \subsetneq (\CA_{i,X}, \nu_{i,X})$,
it follows from 
Definitions \ref{def:submulti} and \ref{def:Euler} that 
$(\CA_i, \nu_i)$, $(\CA_i', \nu_i') = (\CA_{i-1}, \nu_{i-1})$
and $(\CA_i'', \nu_i^*)$
is the triple with respect to $H_i$.
Therefore, by the construction of the chain in \eqref{eq:seq3} and 
Lemma \ref{lem:euler}(i), we have 
\begin{equation}
\label{eq:seq8}
\left(\left(\CA_{i,X}\right)', \nu_{i,X}'\right) 
= \left((\CA_i')_X, (\nu_i')_X\right) 
=  \left(\CA_{i-1, X}, \nu_{i-1, X}\right). 
\end{equation}
Since $(\CA_i, \nu_i)$ is free by assumption, it follows from 
Theorem \ref{thm:multi} that  $(\CA_{i,X}, \nu_{i,X})$ 
is free for each $i$. 
Consequently, it follows from
Remark \ref{rem:restriction} 
and \eqref{eq:seq8}
that also each restriction 
$((\CA_{i,X})^{H_i}, \nu_{i,X}^*) = (\CA_{i,X}'', \nu_{i,X}^*)$
is free with exponents given by Theorem \ref{thm:add-del}(iii).

Since $(\CA_i^{H_i}, \nu_i^*) = (\CA_i'', \nu_i^*)$ 
is inductively free by assumption
and $r(\CA_i'') < r$, it follows from our induction hypothesis
that the localization 
$((\CA_i'')_X, (\nu_{i}^*)_X)$ is also inductively
free for each $i$.
Thus, thanks to \eqref{eq:seq2} and Lemma \ref{lem:euler}(ii),
\[
(\CA_{i,X}'', \nu_{i,X}^*) = ((\CA_{i,X})'', (\nu_{i,X})^*) 
= ((\CA_i'')_X, (\nu_{i}^*)_X)
\] 
is inductively free for each $i$.

Since the rank of $\CA_{1,X}$ is $1$, 
$(\CA_{1,X}, \nu_{1,X})$ 
is inductively free.
Together with the fact that each of the restrictions 
$(\CA_{i,X}'', \nu_{i,X}^*)$
is also inductively
free for each $i$, a repeated application of the 
addition part of Theorem \ref{thm:add-del} then
shows that  the sequence \eqref{eq:seq3}
is an inductive chain 
of $(\CA_X, \nu_X)$, satisfying Definition \ref{def:indfree},
as claimed.

(ii).
The argument is very similar to the one above.
We argue again by induction on the rank $r(\CA)$.
If $r(\CA) \le 3$, then 
$r(\CA_X) \le 2$ for $X \ne T_\CA$, so
the result follows by 
\cite[Cor.\ 7]{ziegler:multiarrangements}.

So suppose $(\CA, \nu)$ is recursively free of rank $r > 3$
and that the statement holds for 
all recursively free multiarrangements of rank less than $r$.

Since $(\CA, \nu)$ is recursively free, there is a 
recursive chain $(\CA_i, \nu_i)$ of $(\CA, \nu)$, where 
$|\nu_i| = |\nu_{i-1}| \pm 1$ for $i = 1, \ldots, n$, 
and $(\CA_n, \nu_n) = (\CA, \nu)$,
see Remark \ref{rem:recchain}.

Since $X$ is a subspace in $V$, as above, 
we can consider the localization 
$(\CA_{i,X}, \nu_{i,X})$ of each member of the recursive chain,
where again 
$\CA_{i,X}$ is short for $(\CA_i)_X$
and $\nu_{i,X}$ for $\nu_i|_{(\CA_i)_X}$,
cf.~\eqref{eq:seq1}.

Then removing redundant 
terms and reindexing the resulting distinct multiarrangements
if needed, 
we obtain a sequence of multiarrangements starting with the 
empty arrangement 
\begin{equation}
\label{eq:seq5}
\Phi_\ell \ne 
\left(\CA_{1,X}, \nu_{1,X}\right) \subsetneq \left(\CA_{2,X}, \nu_{2,X}\right) 
\subsetneq  \ldots \left(\CA_{m,X}, \nu_{m,X}\right) = \left(\CA_X, \nu_X\right),
\end{equation}
where by construction, at each stage we either 
increase or decrease the multiplicity of a single hyperplane by $1$.

Since $(\CA_i, \nu_i)$ is free by assumption, 
it follows from Theorem \ref{thm:multi} that  $(\CA_{i,X}, \nu_{i,X})$ 
is free for each $i$, and so \eqref{eq:seq5} is a chain of 
free submultiarrangemnts of $(\CA_X, \nu_X)$.

Now fix $i$ and let $H$ be the relevant 
hyperplane in the $i$th step in the sequence \eqref{eq:seq5} above,
i.e., the multiplicity of $H$ is either increased or decreased in this step.
In the first instance, letting 
$(\CA_{i,X}, \nu_{i,X})$, $(\CA_{i,X}', \nu_{i,X}') = (\CA_{i-1,X}, \nu_{i-1,X})$, 
and $(\CA_{i,X}'', \nu_{i,X}^*)$
be the triple with respect to $H$, 
we are in the situation of \eqref{eq:seq8} above.
On the other hand,
if the multiplicity of $H$ is decreased in this step, then
let 
$(\CA_{i-1,X}, \nu_{i-1,X})$, $(\CA_{i-1,X}', \nu_{i-1,X}') 
= (\CA_{i,X}, \nu_{i,X})$, 
and $(\CA_{i-1,X}'', \nu_{i-1,X}^*)$
be the triple with respect to $H$.
In the first instance we argue as in (i) above to 
see that 
$((\CA_{i,X})^H, \nu_{i,X}^*) = (\CA_{i,X}'', \nu_{i,X}^*)$
is free with exponents given by Theorem \ref{thm:add-del}(iii).
In the second case we argue in just the same way to get that 
$((\CA_{i-1,X})^H, \nu_{i-1,X}^*) = (\CA_{i-1,X}'', \nu_{i-1,X}^*)$
is free with exponents given by Theorem \ref{thm:add-del}(iii).

If $(\CA_X, \nu_X)$ is inductively free, then it is 
recursively free and we are done. So we may assume that 
$(\CA_X, \nu_X)$ is not inductively free. 
Then in particular, the sequence \eqref{eq:seq5}
is not an inductive chain.
We claim that \eqref{eq:seq5} is a recursive chain of $(\CA_X, \nu_X)$.
Clearly, the initial part of this sequence is
necessarily a chain of inductively free arrangements 
(one needs to add hyperplanes first before one can start removing them 
again).
Let $k$ be maximal so that 
\begin{equation}
\label{eq:seq6}
\Phi_\ell \ne 
\left(\CA_{1,X}, \nu_{1,X}\right) \subsetneq \left(\CA_{2,X}, \nu_{2,X}\right) 
\subsetneq  \ldots \subsetneq \left(\CA_{k,X}, \nu_{k,X}\right)
\end{equation}
is a sequence of inductively free terms in the 
chain \eqref{eq:seq5}.
Then in particular, $(\CA_{k,X}, \nu_{k,X})$ is inductively 
free, hence recursively free.

Since $(\CA_i'', \nu_i^*)$ is recursively free by assumption
and $r(\CA_i'') < r$, it follows from our induction hypothesis
that the localization 
$((\CA_i'')_X, (\nu_{i}^*)_X)$ is also recursively
free for each $i$.
Thus, thanks to Lemma \ref{lem:euler}(ii),
$((\CA_i)_X)'', \nu_{i,X}^*) = ((\CA_i'')_X, (\nu_{i}^*)_X)$ 
is recursively free for each $i$.

In particular, returning to the sequence \eqref{eq:seq6} and 
the $(k+1)$-st step, 
where we reduce a multiplicity 
for the first time in the chain in \eqref{eq:seq5}, 
it follows from the argument above that 
\[
\exp ((\CA_{k,X})^{H_{k+1}}, \nu_{k,X}^*)  \subseteq 
\exp (\CA_{k,X}, \nu_{k,X}).
\]
Therefore, applying the deletion part of Theorem \ref{thm:add-del} and 
using Lemma \ref{lem:euler}(i), it follows that 
\[
\left(\left((\CA_{k,X}\right)', (\nu_{k,X})'\right)  = 
\left((\CA_k')_X, (\nu_{k}')_X\right) = 
\left(\CA_{k+1, X}, \nu_{k+1,X}\right)
\]
is recursively free, where the deletion is with respect to 
$H_{k+1}$.
Now iterate this process. 
\end{proof}

The special case when 
$\nu \equiv 1$ in Theorem \ref{thm:multitwo} gives
Theorem \ref{thm:local}
for the classes $\CIF$ and $\RF$.
Armed with 
Theorem \ref{thm:local} for $\CIF$, we 
obtain the statement of 
Theorem \ref{thm:local} for the class $\HIF$.

\begin{corollary}
\label{cor:heredfree}
Let $U \subseteq V$ be a subspace and let 
$\CA$ be an arrangement in $V$.
If $\CA$ is hereditarily  inductively free,
then so is the localization $\CA_U$.
\end{corollary}

\begin{proof}
As before, for $X = \cap_{H \in \CA_U} H$, we have 
$\CA_X = \CA_U$ and $X \in L(\CA)$. 
Let $Y \in L(\CA_X)$. Then $Y \le X$ in $L(\CA)$.
Since $\CA$ is hereditarily  inductively free,
$\CA^Y$ is inductively free.
So by Theorem \ref{thm:local}
and Lemma \ref{lem:swap}(ii), we get that
$(\CA_X)^Y = (\CA^Y)_X$ is inductively free.
\end{proof}

Theorem \ref{thm:local} thus follows from 
Theorem \ref{thm:multitwo} and 
Corollary \ref{cor:heredfree}.

\begin{remark}
\label{rem:proof}
It is worth noting that the proof
of Theorem \ref{thm:multitwo} shows that any given 
inductive (resp.\ recursive) chain of the ambient 
multiarrangement descends to give an 
inductive (resp.\ recursive) chain of any 
localization. 
\end{remark}

\subsection{Products of inductively free and 
recursively free arrangements}
\label{ssec:recfreeproducts}

Thanks to \cite[Prop.\ 4.28]{orlikterao:arrangements},
the product of two arrangements is free if and only if 
each factor is free. 
In \cite[Prop.\ 2.10]{hogeroehrle:indfree}, 
the first two authors showed that this factorization property
descends to the class of inductively free arrangements.

Let $(\CA_i, \nu_i)$ be a multiarrangement in $V_i$ for $i = 1,2$.
Then the product $(\CA : = \CA_1 \times \CA_2, \nu)$ is a multiarrangement
in $V = V_1 \oplus V_2$ with multiplicity $\nu := \nu_1 \times \nu_2$,
see \cite{abeteraowakefield:euler}. 

The following is just a reformulation of Theorem \ref{thm:recfreep}.

\begin{theorem}
\label{thm:recfreeproductsmulti}
Let $(\CA_i, \nu_i)$ be a multiarrangement in $V_i$ for $i = 1,2$.
Then the product $(\CA, \nu)$ is inductively free 
(resp.\ recursively free) 
if and only if each 
factor $(\CA_i, \nu_i)$ is inductively free 
(resp.\ recursively free). 
\end{theorem}

\begin{proof}
We just give the argument for the case of recursive 
freeness, the argument for inductive freeness is identical.

The reverse implication is straightforward, 
cf.~\cite[Prop.\ 2.10]{hogeroehrle:indfree}.
For the forward implication, assume that 
$(\CA, \nu)$ is recursively free.
Set $X_1 := T_{\CA_1} \oplus V_2$
and $X_2 := V_1 \oplus T_{\CA_2}$. 
Then both $X_1$ and $X_2$ belong to 
the intersection lattice of $\CA$, 
\cite[Prop.\ 2.14]{orlikterao:arrangements}.
Note that 
$\CA_{X_1} = \{H_1 \oplus V_2 \mid H_1 \in \CA_1\} \cong \CA_1$ and 
$\CA_{X_2} = \{V_1 \oplus H_2 \mid H_2 \in \CA_2\} \cong \CA_2$.
It thus follows from Theorem \ref{thm:localmulti} 
that both
$(\CA_{X_1}, \nu_{X_1}) = (\CA_1, \nu_1)$ and 
$(\CA_{X_2}, \nu_{X_2}) = (\CA_2, \nu_2)$ 
are recursively free.
\end{proof}

The special case of Theorem \ref{thm:recfreeproductsmulti} for $\CRFM$
when $\nu_i \equiv 1$ gives Theorem \ref{thm:recfreeproducts}.

\section{Applications to Reflection Arrangements}
\label{sec:appl}

\subsection{Inductive Freeness of Reflection Arrangements}
\label{ssec:reflection}

In this section we demonstrate how 
Theorem \ref{thm:local} can be used to 
show that certain arrangements are not inductively free.

Let $W$ be one of the exceptional complex reflection groups 
$G_{29}$, $G_{33}$,  or $G_{34}$.
Then by 
\cite[Tables C.10, C.14, C.15]{orlikterao:arrangements},
$W$ admits a parabolic subgroup $W_X$ of type 
$G(4,4,3)$, $G(3,3,4)$, or $G(3,3,5)$, respectively.
Thanks to  \cite[Prop.\ 3.2]{hogeroehrle:indfree}, the 
reflection arrangement of 
$G(r,r, \ell)$ is not inductively free 
for $r, \ell \ge 3$.
By \cite[Cor.\ 6.28]{orlikterao:arrangements},
we have
$\CA(W_X) = \CA(W)_X$ and so it follows from 
Theorem \ref{thm:local}, that $\CA(W)$
is not inductively free in each of the three instances.
This was proved in 
\cite[\S 3.1.4]{hogeroehrle:indfree} by different means.
 
\subsection{Inductive Freeness 
of Ziegler's canoncial Multiplicity 
and concentrated Multiplicities 
for monomial groups}
\label{ssec:grrl}

Theorem \ref{thm:localmulti} is very useful in 
showing that a given 
multiarrangement is not inductively free by exhibiting a 
suitable localization which is known to not be inductively free. 
We demonstrate this in the following results.

Let $\CA = \CA(W)$ be the reflection arrangement of 
the complex reflection group $W := G(r,r, \ell)$
for $r, \ell \ge 3$.
Let $H_{i,j}(\zeta) := \ker(x_i - \zeta x_j) \in \CA$,
where $1 \le i < j \le \ell$ and $\zeta$ is an $r$th root of unity
and 
let $H_i :=\ker x_i$ be the $i$th coordinate hyperplane for $1 \le i \le \ell$,
\cite[\S 6.4]{orlikterao:arrangements}.

Fix $H_0 \in \CA$, $m_0 \in \BBZ_{> 1}$ and let 
$\delta = \delta_{H_0,m_0}$ be
as in Definition \ref{def:single}.
Then, since $\CA$ is free
(cf.~\cite[\S 6.3]{orlikterao:arrangements}) so is 
$(\CA, \delta)$, by 
Proposition \ref{prop:single}(i).
Let $\CA''$ be the restriction of $\CA$ with respect 
to $H_0$.
Then $(\CA'', \delta^*) =  (\CA'', \kappa)$ is free, 
thanks to Proposition \ref{prop:single}(ii),
where $\kappa$ is the canonical multiplicity
from Definition \ref{def:zieglermulti}.

Using results from \cite{hogeroehrle:indfree}
and Theorem \ref{thm:localmulti}, 
we show that both $(\CA, \delta)$ 
and $(\CA'', \delta^*)$
fail to be inductively free 
for $\ell \ge 5$.

\begin{proposition}
\label{prop:grrl}
Let $\CA = \CA(W)$ be the reflection arrangement of 
$W = G(r,r, \ell)$ for $r \ge 3, \ell \ge 5$.
Fix $H_0 \in \CA$, $m_0 \in \BBZ_{> 1}$ and let 
$\delta = \delta_{H_0,m_0}$ be as above. Then 
\begin{itemize}
\item[(i)] 
$(\CA, \delta)$ is not inductively free, and
\item[(ii)] 
$(\CA'', \delta^*)$ is not inductively free. 
\end{itemize}
\end{proposition}

\begin{proof}
Since $W$ is transitive on $\CA$, without loss, we may choose
$H_0 := H_{1,2}(1) = \ker(x_1 - x_2)$.
Define 
\[
X := \bigcap_{3 \le i < j \le \ell}{H_{i,j}(\zeta)}
= \bigcap_{3 \le i \le \ell} H_i.
\]
Then $X$ is of rank $\ell-2$ in $L(\CA)$.

First we show (i). 
By definition of $(\CA, \delta)$ and 
\cite[Cor.\ 6.28]{orlikterao:arrangements},
we have
$\CA(W_X) = \CA(W)_X \cong \CA(G(r,r,\ell-2))$ and 
$\delta_X \equiv 1$.
Consequently, $\left( \CA_X, \delta_X\right)$ 
is isomorphic to the simple reflection arrangement
$\CA(G(r,r,\ell-2))$.
Thanks to \cite[Prop.\ 3.2]{hogeroehrle:indfree}, 
the latter is not inductively free, as $\ell \ge 5$.
Therefore, $(\CA, \delta)$
is not inductively free, 
owing to Theorem \ref{thm:localmulti} and so (i) holds.

Now for (ii), recall that
$(\CA'', \delta^*) =  (\CA'', \kappa)$ is free, 
thanks to Proposition \ref{prop:single}(ii).

Set $Y_{i,j}(\zeta) := H_0 \cap H_{i,j}(\zeta) \in \CA''$.
Then one readily checks that for $Y \in \CA''$, we have
\begin{equation}
\label{eq:kappa}
\kappa(Y) = 
\begin{cases}
r-1 & \text{ for } Y = Y_{1,2}(\zeta),\\
2   & \text{ for } Y = Y_{1,i}(\zeta), Y_{2,i}(\zeta) \text{ and } 3 \le i \le \ell, \\
1   & \text{ for } Y = Y_{i,j}(\zeta) \text{ and }  3 \le i < j \le \ell.\\
\end{cases}
\end{equation}
According to \eqref{eq:kappa}, the multiplicity 
$\kappa_X$ of the localization  
$\left( (\CA'')_X, \kappa_X\right)$ satisfies $\kappa_X \equiv 1$.
Thus, $\left( (\CA'')_X, \kappa_X\right)$ 
is isomorphic to the simple reflection arrangement
$\CA(G(r,r,\ell-2))$.
Again by \cite[Prop.\ 3.2]{hogeroehrle:indfree}, 
the latter is not inductively free, as $\ell \ge 5$.
Therefore, 
$(\CA'', \delta^*) = (\CA'', \kappa)$,
is not inductively free, 
thanks to Theorem \ref{thm:localmulti} and (ii) follows.
\end{proof}

Proposition \ref{prop:grrl} 
generalizes to a larger class of multiarrangements
stemming from complex reflection groups.
Orlik and Solomon defined 
complex $\ell$-arrangements $\CA^k_\ell(r)$ in 
\cite[\S 2]{orliksolomon:unitaryreflectiongroups}
(cf.\ \cite[\S 6.4]{orlikterao:arrangements}) which
interpolate between the
reflection arrangements of the complex reflection groups
$G(r,r,\ell)$ and $G(r,1,\ell)$. 
For $r, \ell \geq 3$ and $0 \leq k \leq \ell$ 
the defining polynomial of
$\CA^k_\ell(r)$ is given by
\[
Q(\CA^k_\ell(r)) = x_1 \cdots x_k\prod\limits_{1 \le i < j \le \ell}(x_i^r - x_j^r),
\]
so that 
$\CA^\ell_\ell(r) = \CA(G(r,1,\ell))$ and 
$\CA^0_\ell(r) = \CA(G(r,r,\ell))$. 
Again fix $H_0 \in \CA$, $m_0 \in \BBZ_{> 1}$ and let 
$\delta = \delta_{H_0,m_0}$ be
as in Definition \ref{def:single}.
Then, since $\CA$ is free
(cf.~\cite[Prop.\ 6.85]{orlikterao:arrangements}),
so is $(\CA, \delta)$, by 
Proposition \ref{prop:single}(i).
Let $\CA''$ be the restriction of $\CA$ with respect 
to $H_0$.
Then $(\CA'', \delta^*) =  (\CA'', \kappa)$ is free, 
thanks to Proposition \ref{prop:single}(ii).

Combining results from \cite{hogeroehrle:indfree}
with Theorem \ref{thm:localmulti}, 
we show that both $(\CA, \delta)$ 
and $(\CA'', \delta^*)$
fail to be inductively free 
provided 
$\ell \ge 5$, $0 \le k \le \ell-3$
and $H_0$ is of the form $H_{i,j}(\zeta)$
in the latter case.

\begin{proposition}
\label{prop:akl}
Let $\CA = \CA^k_\ell(r)$ for $r \ge 3, \ell \ge 5$
and $0 \le k \le \ell-3$.
Fix $H_0 \in \CA$, $m_0 \in \BBZ_{> 1}$ and let 
$\delta = \delta_{H_0,m_0}$ be as above. Then 
\begin{itemize}
\item[(i)] 
$(\CA, \delta)$ is not inductively free, and
\item[(ii)] 
for $H_0 = H_{i,j}(\zeta)$, also 
$(\CA'', \delta^*)$ is not inductively free. 
\end{itemize}
\end{proposition}

\begin{proof}
For $k = 0$, this is just Proposition \ref{prop:grrl}. 
So we may assume that $1 \le k \le \ell-3$.

First we show (i). Define 
\[
Z := \bigcap_{\ell-2 \le i < j \le \ell}{H_{i,j}(\zeta)}
= \bigcap_{\ell-2 \le i \le \ell} H_i.
\]
Then $Z$ is of rank $3$ in $L(\CA)$.
Without loss 
we may suppose that either
$H_0 := H_{1,2}(1) = \ker(x_1 - x_2)$
or 
$H_0 := H_1 = \ker x_1$.
In both instances, 
by the definition of $(\CA, \delta)$
and the fact that
$\ell \ge 5$
and $0 \le k \le \ell-3$,
we have $\delta_Z \equiv 1$, and so 
$(\CA_Z, \delta_Z)$ is isomorphic to the simple 
reflection arrangement $\CA(G(r,r,3))$.
By \cite[Prop.\ 3.2]{hogeroehrle:indfree}, 
the latter is not inductively free.
Therefore, $(\CA, \delta)$
is not inductively free, by
Theorem \ref{thm:localmulti}, so that (i) holds.

For (ii), recall again that
$(\CA'', \delta^*) =  (\CA'', \kappa)$ is free, 
thanks to Proposition \ref{prop:single}(ii).
We may suppose without loss that 
$H_0 := H_{1,2}(1) = \ker(x_1 - x_2)$.
Set $Y_{i,j}(\zeta) := H_0 \cap H_{i,j}(\zeta)$ and 
$Y_i  := H_0 \cap H_i$ in $\CA''$.
Then one readily checks that for $Y \in \CA''$, we have
\begin{equation}
\label{eq:akl}
\kappa(Y) = 
\begin{cases}
r+1\  (\text{resp. } r) & \text{ for } Y = Y_{1,2}(\zeta) 
\text{ and } k \ge 2\  (\text{resp. }  k = 1),\\
2   & \text{ for } Y = Y_{1,i}(\zeta), Y_{2,i}(\zeta) \text{ and } 3 \le i \le \ell,\\
1   & \text{ for } Y = Y_{i,j}(\zeta) \text{ and }  3 \le i < j \le \ell,\\
1   & \text{ for } Y = Y_i \text{ and }  3 \le i \le k,\\
\end{cases}
\end{equation}
where the value of $\kappa(Y)$ in the first case depends on $k$ and 
the last instance only occurs if $k \ge 3$. 
Define 
\[
X := H_0 \cap Z = H_0 \cap \left(\bigcap_{\ell-2 \le i \le \ell} H_i\right).
\]
Then $X$ is of rank $3$ in $L(\CA'')$.
According to \eqref{eq:akl}, the multiplicity 
$\kappa_X$ of the localization  
$\left( (\CA'')_X, \kappa_X\right)$ satisfies $\kappa_X \equiv 1$.
Thus, it follows from the construction 
and the fact that $\ell \ge 5$
and $0 \le k \le \ell-3$, that the localization 
$((\CA'')_X, \kappa_X)$ is isomorphic to the simple 
reflection arrangement $\CA(G(r,r,3))$.
Once again by \cite[Prop.\ 3.2]{hogeroehrle:indfree}, 
the latter is not inductively free.
Therefore, $(\CA'', \delta^*) =  (\CA'', \kappa)$
is not inductively free, by
Theorem \ref{thm:localmulti}, so (ii) follows.
\end{proof}




\bigskip

\bibliographystyle{amsalpha}

\newcommand{\etalchar}[1]{$^{#1}$}
\providecommand{\bysame}{\leavevmode\hbox to3em{\hrulefill}\thinspace}
\providecommand{\MR}{\relax\ifhmode\unskip\space\fi MR }
\providecommand{\MRhref}[2]{%
  \href{http://www.ams.org/mathscinet-getitem?mr=#1}{#2} }
\providecommand{\href}[2]{#2}


\end{document}